\documentclass[12pt,english,reqno]{amsart}  

\usepackage[T1]{fontenc} % Ensures proper font encoding
\usepackage{geometry} % Controls margins
\usepackage[utf8]{inputenc} % Input encoding
\usepackage{amssymb,amsmath,amsthm} % Core mathematical packages

\usepackage{graphicx}
\usepackage{tikz} % Used for diagrams
\usetikzlibrary{arrows,automata} % Load specific TikZ libraries if needed

\usepackage{hyperref}
\hypersetup{
    colorlinks=true,
    linkcolor=blue,
    filecolor=magenta,      
    urlcolor=cyan,
}

\newtheorem{theorem}{Theorem}[section]
\newtheorem{corollary}[theorem]{Corollary}
\newtheorem{lemma}[theorem]{Lemma}

\newtheorem{definition}[theorem]{Definition}
\newtheorem{remark}[theorem]{Remark}

\newcommand{\bfS}[1]{\ensuremath{\mathbf{S}^{#1}}} % Sphere notation
\newcommand{\floor}[1]{\lfloor #1 \rfloor} % Floor
\newcommand{\R}{\mathbb R}

\newcommand{\mymatrixstretch}{1.3}

\title[CMC Hypertorus in \bfS{4}]{Existence of a CMC hypertorus in \bfS{4} via computer assistance}

\author{Oscar Perdomo}
\address{Department of Mathematical Sciences, Central Connecticut State University}
\email{perdomoosm@ccsu.edu}

\date{\today}

\begin{document}

\maketitle

\begin{abstract}
The round Taylor method uses rational arithmetic, allowing control of both the round-off error and the numerical error in approximating solutions of differential equations. In this paper, we employ this method together with the Poincaré–Miranda theorem to prove the existence of a hypertorus with constant mean curvature (CMC) in the four-dimensional unit sphere.
\end{abstract}

\section{Introduction and Preliminary Theorems}

The simplest examples of compact constant mean-curvature (CMC) hypersurfaces in the unit sphere \(\bfS{n+1}(1)\) are the totally umbilical spheres \(S^n(r)\), the Clifford hypersurfaces 
\[
  \bfS{k}(r)\times \bfS{n-k}\bigl(\sqrt{1-r^2}\bigr),
\] 
and the rotational hypersurfaces \cite{P2010}, which can be parametrized by
\[
  \phi(t,y)
  \;=\;
  \bigl(\xi_2(t)\,y,\;\xi_1(t)\bigr),
  \quad
  y\in \bfS{n-1}(1),\; t\in\R.
\]

More recently, Carlotto and Schulz \cite{Carl} exhibited several minimal hypersurfaces in \(\bfS{2n+2}(1)\) of the form

\[
  \phi(t,y,z)
  \;=\;
  \bigl(\xi_3(t)\,z,\;\xi_2(t)\,y,\;\xi_1(t)\bigr),
  \quad
  y,z\in \bfS{n}(1),\; t\in\R.
\]
One of their examples was extended to a family of CMC hypersurfaces in \(\bfS{2n+2}(1)\) by Wei et al.\ \cite{HW,LW}.  For minimal hypersurfaces, in the  case \(n=1\), it has been conjectured that the embedded example of Carlotto–Schulz is the only minimal embedded hypertorus (an immersion of \(\bfS{1}\times \bfS{1}\times \bfS{1}\)) in \(\bfS{4}\).  Very recently, the author has numerically explored other families of CMC hypersurfaces of the same type and found evidence for new compact embedded examples, distinct from those of \cite{HW,LW}.

In this paper, we combine a rational-arithmetic ``Round Taylor'' method \cite{Pna1} with the Poincaré–Miranda theorem to give a rigorous existence proof of a new embedded CMC hypersurface in \(\bfS{4}\).  In particular, we address the following issues:
\begin{itemize}
  \item Numerical results may depend sensitively on the computer’s rounding strategy.
  \item As noted in \cite{Pna1}, if one avoids any rounding, then after just thirty steps of the Euler method for
\[
  y' = y - \tfrac13y^2,\quad y(0)=\tfrac12,\quad \Delta t=0.01
\]
the exact result is the rational
\[
  \frac{1244197046778066277036445468762843519}
       {2407347121029120000000000000000000000}.
\]
The numerator and denominator grow so large that they become impractical for standard numerical evaluation.

  \item We must ensure that the true solution of the ODE exists up to the target time; naive numerical methods can produce an ``answer'' at some \(t_*\) even if the actual solution ceases to be defined there.

\end{itemize}

Let us state the our main theorem

\begin{theorem}\label{mainthm}

There exist real numbers $a_*\in\left(\frac{5204}{10000},\frac{5244}{10000}\right)$ and  
$t_*\in\left(\frac{3966}{10000},\frac{3991}{10000}\right)$ such that the solution 
$(r(t),\theta(t),\alpha(t))$ of the differential  equation

\[
\left\{
\begin{aligned}
\dot r     &= \cos\alpha,\\
\dot \theta &= \frac{\sin\alpha}{\sin r},\\
\dot \alpha &= \frac{2\cot(2\theta)}{\sin r}\,\cos\alpha
               - 3\cot r\,\sin\alpha -3\,,
\end{aligned}
\right.
\]

that satisfies $r(0)=\frac{\pi}{2}$, $\theta(0)=a_*$ and $\alpha(0)=\pi$ is defined for all $t\in[0,\frac{3991}{10000}]$ and 

$$\alpha(t_*)=\frac{\pi}{2},\quad \theta(t_*)=\frac{\pi}{4}$$
\end{theorem}

Due to the symmetries of the differential equation, it follows that the solution $(r(t),\theta(t),\alpha(t))$ is defined on the interval $[0,4t_*]$ (actually it is define for all real numbers) and 

\begin{eqnarray}\label{profile curve}
 \bigl(\sin r(t) \cos \theta(t) ,\sin r(t) \sin \theta(t),\cos r(t)   \bigr)
\end{eqnarray}
defines a smooth closed curve in the sphere $\bfS{2}$. See for example  \cite{Carl, HW}. We have that $\beta(0)=\beta(4t_*)$ and $\beta'(0)=\beta'(4t_*)$.

If follows from the work in \cite{HW} that the curve $\beta(t)$ is the profile curve of an embedded constant mean curvature hypersurface in $\bfS{4}$. More precisely, we have:

\begin{corollary} If $(r(t),\theta(t),\alpha(t))$ is the curve defined in Theorem \ref{mainthm}, then 

$$M=\{ \bigl(\sin r(t) \cos \theta(t)\, y, \sin r(t)  \sin \theta(t) z,\cos r(t)\bigr): t\in [0,4t_*], y, z\in \bfS{1}\}$$  

is an embedded  hypersurface in $\bfS{4}$ with constant mean curvature $H=-3$

\end{corollary}

Let us finish the introduction by stating two theorems which along with the Round Taylor method are essentially all that is needed to show the main theorem in this paper.

The Poincaré–Miranda theorem can be thought of as an extension of the intermediate value theorem for a system of two functions with two variables, it states: 

\begin{theorem}[Poincaré–Miranda in two variables]\label{thm:poincare-miranda-2d}
Let 
\[
  F,\,G :   [a_1,a_2]\times[t_1,t_2]. \;\longrightarrow\;\mathbb{R}
\]
be continuous.  If
\begin{enumerate}%[\upshape(i)]
  \item \(F(a,t_1)>0\) and \(F(a,t_2)<0\) for all \(a\in[a_1,a_2]\),
  \item \(G(a_1,t)<0\) and \(G(a_2,t)>0\) for all \(t\in[t_1,t_2]\),
\end{enumerate}
then there exists a point \((a_0,t_0)\in[a_1,a_2]\times[t_1,t_2]\) such that
\[
  F(a_0,t_0)=0
  \quad\text{and}\quad
  G(a_0,t_0)=0.
\]
\end{theorem}

The second theorem \cite{coddington1989}, is a consequence of  Gronwall's inequality and allows us to compare two different solutions of the same differential equation.

\begin{theorem}\label{thm:stability}
Let us assume that 
\[
f\colon U\subset\mathbb{R}^n \;\longrightarrow\;\mathbb{R}^n
\]
is \(C^1\) functions defined on an open convex set \(U\) such that
\(\lvert Df(u)\rvert<K\).  If \(y(t)\) and \(z(t)\) satisfy
\(\dot y(t)=f(y(t))\) and \(\dot z(t)=f(z(t))\), respectively, then
\[
\lvert y(t)-z(t)\rvert
\le
\lvert y(t_0)-z(t_0)\rvert\,e^{K\lvert t-t_0\rvert}.
\]
\end{theorem}

%%%%%%%%%%%%
%%%%%%%%%%%%
%%%%%%%%%%%%

\section{The Round Taylor method and some bounds}

The Round Taylor Method (RTM) of order \(m\) is a small variation of the classical Taylor method for approximating solutions of initial-value ODEs. Its first advantage is that each step produces a well-defined rational vector, so the result is independent of the computer’s floating-point behavior. Its second advantage is that the global-error estimate theorem incorporates the rounding error.  Its third advantage is that, although verifying the theorem’s hypotheses can be intricate, it provides a fully rigorous existence proof for the true solution of the differential equation studied.

RTM relies on the floor function \(\floor{x}\), which assigns to a real number \(x\) the largest integer not exceeding \(x\).  For a vector \(x=(x_1,\dots,x_k)\), we write
\[
  \floor{x} \;=\; (\floor{x_1},\dots,\floor{x_k})\,.
\]

Consider the initial‐value problem
\begin{equation}\label{ns}
  Y'(t) \;=\; f\bigl(Y(t)\bigr), 
  \quad
  Y(0)=y_0,
\end{equation}
where \(Y(t)=(Y_1(t),\dots,Y_n(t))^T\) and \(f(y)=(f_1(y),\dots,f_n(y))^T\).  Given two positive rationals \(h\) and \(R\), we define sequences \(\{z_i\}\) and \(\{y_i\}\subset\R^n\) by
taking $z_0=R\;\floor{\tfrac1R\,y_{i+1}}$ and then
\begin{align}\label{tm}
y_{i+1}
&= z_i + f(z_i)\,h
   + F_1(z_i)\,\frac{h^2}{2!}
   + \dots
   + F_{m-1}(z_i)\,\frac{h^m}{m!},
\notag\\
z_{i+1}
&= R\;\floor{\tfrac1R\,y_{i+1}},
\end{align}
starting from \(z_0=y_0\).  Here
\[
  F_1=(Df)\,f,\quad
  F_2=(DF_1)\,f,\;\dots,\;
  F_{m-1}=(DF_{m-2})\,f,
\]
and \(D\) denotes the Jacobian operator.

\begin{theorem}\label{RTM}
Let \(h,R>0\) and \(k\in\mathbb{N}\).  Consider the RTM sequences \(\{(y_i,z_i)\}_{i=0}^k\) from \eqref{tm} for the ODE \(Y'=f(Y)\), \(Y(0)=y_0\).  Suppose there exist constants \(M_0,M_1,\dots,M_n\), \(K_0,\dots,K_{m-1}\) and 
two sets
\[
  U_1 = \bigl\{u\in\R^n:b_i\le u_i\le c_i,\;i=1,\dots,n\bigr\},
  \quad
  U_2 = \bigl\{u\in\R^n:b_i-\epsilon< u_i< c_i+\epsilon,\;i=1,\dots,n\bigr\},
\]
with
\[
  \epsilon > M_0\,h + \tilde R,
  \quad
  \tilde R
  = \Bigl(\tfrac{M\,h^m}{L\,(m+1)!} + \tfrac{\sqrt n\,R}{L\,h}\Bigr)
    \bigl(e^{Lkh}-1\bigr),
\]
where \(M=\sqrt{M_1^2+\cdots+M_n^2}\) and 
\[
  L = K_0 + K_1\,\tfrac h{2!} + \dots + K_{m-1}\,\tfrac{h^{m-1}}{(m-1)!}.
\]
If
\begin{itemize}
  \item \(f\) and all its partials up to order \(m+1\) are continuous on \(\overline{U_2}\);
  \item \(z_j\in U_1\) for \(j=0,1,\dots,k\);
  \item \(\lvert (F_m)_i(u)\rvert\le M_i\) for \(i=1,\dots,n\) and all \(u\in U_2\);
  \item \(\lvert Df(u)\rvert\le K_0\) and \(\lvert DF_i(u)\rvert\le K_i\) for \(i=1,\dots,m-1\) on \(U_2\);
  \item \(\lvert f_i(u)\rvert\le M_0\) for \(i=1,\dots,n\) on \(U_2\),
\end{itemize}
then the true solution \(Y(t)\) exists on \([0,hk]\), remains in \(U_2\), and satisfies
\[
  \lvert Y(jh) - z_j\rvert \;\le\; \tilde R,
  \quad
  j=1,2,\dots,k.
\]
\end{theorem}

Let us continue with a series of Lemmas and definitions necessary to check the conditions of Theorem \ref{RTM}.

\begin{definition}\label{def1} Let us define the sets 

\[
  U_1 = \bigl\{u\in\R^3:b_i\le u_i\le c_i,\;i=1,2,3\bigr\},
  \quad
  U_2 = \bigl\{u\in\R^3:b_i-\epsilon< u_i< c_i+\epsilon,\;i=1,2,3\bigr\},
\]

where 

$$(b_1,c_1,b_2,c_2,b_3,c_3)=\left(\frac{1321}{1000},\frac{1571}{1000},\frac{261}{500},\frac{393}{500},\frac{157}{100},\frac{1571}{500}\right)$$

and  $\epsilon=\frac{1}{10^3}$. Let us also define the functions

 $$g_1=\cos \left(u_3\right),g_2=\sin \left(u_3\right),g_3=\csc \left(u_1\right),g_4=\cot \left(u_1\right),g_5=\cot \left(2 u_2\right),g_6=\csc ^2\left(2 u_2\right)$$
  and the vector field 
 
 $$f=\left(f_1,\, f_2,\, f_3\right)=\left(g_1,\, g_2 g_3,\, -3 g_2 g_4+2 g_1 g_3 g_5-3\right)$$ 
  
\end{definition}

\begin{remark}
If we identify the function $r$ with the variable $u_1$, the function $\theta$ with the variable $u_2$ and the function $\alpha$ with the variable $u_3$, that is, if we identify $(r,\theta,\alpha)^T$ with $Y=(u_1,u_2,u_3)^T$, then the differential equation in Theorem \ref{mainthm} can be written as 

$$\dot{Y}=f(Y)$$
\end{remark}

%%%
%%%
\begin{lemma} \label{lemma1} If $g_i$, $i=1,\dots, 6$ and $U_2$ are defined as in Definition \ref{def1}, then

% somewhere before your big display:
\renewcommand{\arraystretch}{1.3}  % 1 is default; 1.3 gives ~30% more line spacing

\[
  \begin{array}{r@{\;}c@{\;}l@{\;}c@{\;}l@{\;}c@{\;}l}
    d_1 &=& -1 
       &\le& g_1 
       &\le& \cos\bigl(\tfrac{1569}{1000}\bigr)=e_1, \\[0.75ex]
    d_2 &=& \sin\bigl(\tfrac{3143}{1000}\bigr) 
       &\le& g_2 
       &\le& 1=e_2,        \\[0.75ex]
    d_3 &=& 1 
       &\le& g_3 
       &\le& \csc\bigl(\tfrac{33}{25}\bigr)=e_3, \\[0.75ex]
    d_4 &=& \cot\bigl(\tfrac{393}{250}\bigr) 
       &\le& g_4 
       &\le& \cot\bigl(\tfrac{33}{25}\bigr)=e_4, \\[0.75ex]
    d_5 &=& \cot\bigl(\tfrac{787}{500}\bigr) 
       &\le& g_5 
       &\le& \cot\bigl(\tfrac{521}{500}\bigr)=e_5, \\[0.75ex]
    d_6 &=& 1 
       &\le& g_6 
       &\le& \csc^2\bigl(\tfrac{521}{500}\bigr)=e_6.
  \end{array}
\]

\end{lemma}

\begin{proof} The proof is straighforward because all the function $g_i$ given by expression  in only one variable. Notice that

$$\left(b_1-\epsilon,c_1+\epsilon,b_2-\epsilon,c_2+\epsilon,b_3-\epsilon,c_3+\epsilon\right)=\left(\frac{33}{25},\frac{393}{250},\frac{521}{1000},\frac{787}{1000},\frac{1569}{1000},\frac{3143}{1000}\right).$$

\end{proof}

\begin{lemma}\label{lemma2}
For any $u\in U_2$, the functions $f_1,f_2$ and $f_3$ satisfies that 

$$|f_1|\le 1, \quad |f_2|<1.033,\, |f_3|<4.98$$
\end{lemma}

\begin{proof}
The bound on $f_1$ follows directly from Lemma \ref{lemma1}. To prove the bound on $f_2=\sin(u_3)\csc(u_1)=g_2g_3$ we define the function $g_7=w_2w_3$ with $w_2\in [d_2,e_2]$ and  $w_3\in [d_3,e_3]$ a direct computation shows that in this domain we have that 

$$-0.002< \sin \left(\frac{3143}{1000}\right) \le g_7\le\csc \left(\frac{33}{25}\right),\csc \left(\frac{33}{25}\right)<1.033$$

In order to show the bound on $f_3$ we define the function $g_8=2w_1w_3w_5$ with $w_1\in [d_1,e_1]$, $w_3\in [d_3,e_3]$ and  $w_5\in [d_5,e_5]$ and the function 
$g_9=-3-3w_2w_4$ with $w_2\in [d_2,e_2]$ and  $w_4\in [d_4,e_4]$ a direct computation shows that 

\begin{eqnarray*}
-1.2064<-2 \cot \left(\frac{521}{500}\right) \csc \left(\frac{33}{25}\right)\le &g_8&\le\ -2 \cot \left(\frac{787}{500}\right) \csc \left(\frac{33}{25}\right)<0.0067\\
-3.7686<-3-3 \cot \left(\frac{33}{25}\right)\le &g_9&\le\ -3-3 \cot \left(\frac{393}{250}\right)<-2.9964
\end{eqnarray*}

using the triangular inequality it follows that $|f_3|<1.2064+3.7686<4.98$

\end{proof}

\begin{lemma}\label{lemma3} The Jacobian of the vector field $f$ is given by
$$Df=\left(
\begin{array}{ccc}
 0 & 0 & -g_2 \\
 -g_2 g_3 g_4 & 0 & g_1 g_3 \\
 g_3 \left(3 g_2 g_3-2 g_1 g_4 g_5\right) & -4 g_1 g_3 g_6 & -3 g_1 g_4-2 g_2 g_3 g_5 \\
\end{array}
\right)=\left\{\frac{\partial f_i}{\partial u_j}\right\}_{ij}$$

Moreover, if we define the matrix 

$$bDf=\left(
\begin{array}{ccc}
 0 & 0 & 1 \\
 0.265 & 0 & 1.033 \\
 3.506 & 5.539 & 1.21 \\
\end{array}
\right)=\left\{b_{ij}\right\}_{ij}$$

Then for any $u\in U_2$ and any $1\le i,j\le 3$ we have

$$\left|\frac{\partial f_i}{\partial u_j}(u)\right|\le b_{ij}$$

and $|Df(u)|<6.8246=K_0$.

\end{lemma}

\begin{proof}
The proof follows the same arguments as those in  the proof of Lemma \ref{lemma2}. 
\end{proof}

\begin{lemma}\label{lemma4} If we define $F_1=\left(F_{11},F_{12},F_{13} \right)=Df\cdot f$ then for any $u\in U_2$ we have that

$$F_{11}<M_1=4.98,\quad F_{12}<M_2=5.41,\quad F_{13}<M_3=15.26$$

We also have that for any $u\in U_2$ and $i=1,2,3$,

$$|f_i|<4.98$$
\end{lemma}

\begin{proof}
From Lemma \ref{lemma1} we have that $|f_1|\le 1, \quad |f_2|<1.033,\, |f_3|<4.98$. The proof of this Lemma uses basic arguments combining Lemma \ref{lemma2}, the triangular inequality and the fact that

$$\left(
\begin{array}{ccc}
 0 & 0 & 1 \\
 0.265 & 0 & 1.033 \\
 3.506 & 5.539 & 1.21 \\
\end{array}
\right)\begin{pmatrix}1\\1.033\\4.98\end{pmatrix}=\begin{pmatrix}4.98\\5.40934\\15.253587\end{pmatrix}
$$
\end{proof}
%%%%%
%%%%
%%%%%
\section{Proof of the main theorem}

In this section we prove Theorem \ref{mainthm}. We will be applying the RTM with \(R=10^{-10}\) (rounding to ten decimals) and, our values for $h$ would be either

\[
  h = 
  \begin{cases}
    \dfrac{0.3966}{25000}, & \text{when estimating at }t=0.3966,\\[6pt]
    \dfrac{0.3991}{25000}, & \text{when estimating at }t=0.3991.
  \end{cases}
\]
The value for $\epsilon$, $M_1,M_2,M_3,K_0$ and the sets $U_1$ and $U_2$ that we will be using for the RTM (Theorem \ref{RTM}) are taken from Definition \ref{def1} and Lemmas \ref{lemma3} and \ref{lemma4}. Since all entries of the sequence \(\{z_i\}\) are monotone, it is not hard to show \(z_j\in U_1\).  For the reader’s convenience we restate the theorem.

\begin{theorem}\label{mainthm}

There exist real numbers $a_*\in\left(\frac{5204}{10000},\frac{5244}{10000}\right)$ and  
$t_*\in\left(\frac{3966}{10000},\frac{3991}{10000}\right)$ such that the solution 
$(r(t),\theta(t),\alpha(t))$ of the differential  equation

\[
\left\{
\begin{aligned}
\dot r     &= \cos\alpha,\\
\dot \theta &= \frac{\sin\alpha}{\sin r},\\
\dot \alpha &= \frac{2\cot(2\theta)}{\sin r}\,\cos\alpha
               - 3\cot r\,\sin\alpha -3\,,
\end{aligned}
\right.
\]

that satisfies $r(0)=\frac{\pi}{2}$, $\theta(0)=a_*$ and $\alpha(0)=\pi$ is defined for all $t\in[0,\frac{3991}{10000}]$ and 

$$\alpha(t_*)=\frac{\pi}{2},\quad \theta(t_*)=\frac{\pi}{4}$$
\end{theorem}

\begin{proof}
Let us consider the functions $\alpha$ and $\theta$ as functions in two variables. More precisely we define $\alpha(a,t)$ and $\theta(a,t)$ as the solution of the differential equation that satisfies the initial condition $r(0)=\frac{\pi}{2}$, $\alpha(0)=\pi$ and $\theta(0)=a$. The idea of the proof is to show that 

$$\alpha(a,0.3966)>\frac{\pi}{2} \quad\hbox{and}\quad \alpha(a,0.3991)<\frac{\pi}{2}\quad\hbox{for all $a\in[0.5204,0.5244]$}$$
and
$$\theta(0.5204,t)<\frac{\pi}{4} \quad\hbox{and}\quad \theta(0.5244,t)>\frac{\pi}{4}\quad\hbox{for all $t\in[0.3966,0.3991]$}.$$

The theorem will follow from the Poincaré–Miranda theorem.  To show the inequalities for \(\theta(a,t)\) is relatively easy because we have explicit bounds on $\frac{\partial \theta}{\partial t}=\theta'(t)$. To show the inequality for the function $\alpha(a,t)$ is a little more challenging because the control on the change of $\alpha$ with respect to the variable $a$ comes from Theorem \ref{thm:stability}. Since the bound given by the Theorem is exponential we need to estimate the function on several values of $a$, to be precise we will be considering 

$$a_i=0.5204+j\, \frac{0.004}{15}\quad\hbox{with }\quad j=0,\dots, 15$$

A direct computation gives us that if we do the RTM to  estimate  $(r(a_j,0.3966),\theta(a_j,0.3966),\alpha(a_j,0.3966))$ using $h=\frac{0.3966}{25000}$ and  $(r(a_j,0.3991),\theta(a_j,0.3991),\alpha(a_j,0.3991))$ using $h=\frac{0.3991}{25000}$ we obtain the values

\[
% locally bump \arraystretch
{\renewcommand{\arraystretch}{\mymatrixstretch}%
\left(
\begin{array}{ccc}
\tilde{r}(a_0,0.3966) & \tilde{\theta}(a_0,0.3966)& \tilde{\alpha}(a_0,0.3966) \\
\tilde{r}(a_1,0.3966) & \tilde{\theta}(a_1,0.3966)& \tilde{\alpha}(a_1,0.3966)  \\
\tilde{r}(a_2,0.3966) & \tilde{\theta}(a_2,0.3966)& \tilde{\alpha}(a_2,0.3966) \\
\tilde{r}(a_3,0.3966) & \tilde{\theta}(a_3,0.3966)& \tilde{\alpha}(a_3,0.3966) \\
\tilde{r}(a_4,0.3966) & \tilde{\theta}(a_4,0.3966)& \tilde{\alpha}(a_4,0.3966) \\
\tilde{r}(a_5,0.3966) & \tilde{\theta}(a_5,0.3966)& \tilde{\alpha}(a_5,0.3966)  \\
\tilde{r}(a_6,0.3966) & \tilde{\theta}(a_6,0.3966)& \tilde{\alpha}(a_6,0.3966) \\
\tilde{r}(a_7,0.3966) & \tilde{\theta}(a_7,0.3966)& \tilde{\alpha}(a_7,0.3966) \\
\tilde{r}(a_8,0.3966) & \tilde{\theta}(a_8,0.3966)& \tilde{\alpha}(a_8,0.3966) \\
\tilde{r}(a_9,0.3966) & \tilde{\theta}(a_9,0.3966)& \tilde{\alpha}(a_9,0.3966) \\
\tilde{r}(a_{10},0.3966) & \tilde{\theta}(a_{10},0.3966)& \tilde{\alpha}(a_{10},0.3966) \\
\tilde{r}(a_{11},0.3966) & \tilde{\theta}(a_{11},0.3966)& \tilde{\alpha}(a_{11},0.3966) \\
\tilde{r}(a_{12},0.3966) & \tilde{\theta}(a_{12},0.3966)& \tilde{\alpha}(a_{12},0.3966) \\
\tilde{r}(a_{13},0.3966) & \tilde{\theta}(a_{13},0.3966)& \tilde{\alpha}(a_{13},0.3966)  \\
\tilde{r}(a_{14},0.3966) & \tilde{\theta}(a_{14},0.3966)& \tilde{\alpha}(a_{14},0.3966) \\
\tilde{r}(a_{15},0.3966) & \tilde{\theta}(a_{15},0.3966)& \tilde{\alpha}(a_{15},0.3966)  \\
\end{array}
\right)
=\left(
\begin{array}{ccc}
 \frac{2644387103}{2000000000} & \frac{488979583}{625000000} & \frac{15734399013}{10000000000} \\
 \frac{1322132297}{1000000000} & \frac{7825948977}{10000000000} & \frac{15737174823}{10000000000} \\
 \frac{6610355281}{5000000000} & \frac{7828224597}{10000000000} & \frac{3934987681}{2500000000} \\
 \frac{3305024603}{2500000000} & \frac{7830500241}{10000000000} & \frac{15742726973}{10000000000} \\
 \frac{1652435797}{1250000000} & \frac{7832775777}{10000000000} & \frac{15745503261}{10000000000} \\
 \frac{13218874699}{10000000000} & \frac{489690711}{625000000} & \frac{7874139927}{5000000000} \\
 \frac{2643652613}{2000000000} & \frac{3918663461}{5000000000} & \frac{7875528331}{5000000000} \\
 \frac{13217651787}{10000000000} & \frac{1567920511}{2000000000} & \frac{1969229197}{1250000000} \\
 \frac{3304260153}{2500000000} & \frac{1960469517}{2500000000} & \frac{1575661087}{1000000000} \\
 \frac{13216429757}{10000000000} & \frac{7844153569}{10000000000} & \frac{630375529}{400000000} \\
 \frac{825988681}{625000000} & \frac{7846429189}{10000000000} & \frac{15762165867}{10000000000} \\
 \frac{412975267}{312500000} & \frac{7848704599}{10000000000} & \frac{1576494373}{1000000000} \\
 \frac{13214598143}{10000000000} & \frac{1570196029}{2000000000} & \frac{7883860961}{5000000000} \\
 \frac{6606994017}{5000000000} & \frac{3926627829}{5000000000} & \frac{15770500083}{10000000000} \\
 \frac{2642675617}{2000000000} & \frac{7855531043}{10000000000} & \frac{15773278497}{10000000000} \\
 \frac{13212768441}{10000000000} & \frac{3928903267}{5000000000} & \frac{7888028663}{5000000000} \\
\end{array}
\right)
}
\]

\[
% locally bump \arraystretch
{\renewcommand{\arraystretch}{\mymatrixstretch}%
\left(
\begin{array}{ccc}
\tilde{r}(a_0,0.3991) & \tilde{\theta}(a_0,0.3991)& \tilde{\alpha}(a_0,0.3991) \\
\tilde{r}(a_1,0.3991) & \tilde{\theta}(a_1,0.3991)& \tilde{\alpha}(a_1,0.3991)  \\
\tilde{r}(a_2,0.3991) & \tilde{\theta}(a_2,0.3991)& \tilde{\alpha}(a_2,0.3991) \\
\tilde{r}(a_3,0.3991) & \tilde{\theta}(a_3,0.3991)& \tilde{\alpha}(a_3,0.3991) \\
\tilde{r}(a_4,0.3991) & \tilde{\theta}(a_4,0.3991)& \tilde{\alpha}(a_4,0.3991) \\
\tilde{r}(a_5,0.3991) & \tilde{\theta}(a_5,0.3991)& \tilde{\alpha}(a_5,0.3991)  \\
\tilde{r}(a_6,0.3991) & \tilde{\theta}(a_6,0.3991)& \tilde{\alpha}(a_6,0.3991) \\
\tilde{r}(a_7,0.3991) & \tilde{\theta}(a_7,0.3991)& \tilde{\alpha}(a_7,0.3991) \\
\tilde{r}(a_8,0.3991) & \tilde{\theta}(a_8,0.3991)& \tilde{\alpha}(a_8,0.3991) \\
\tilde{r}(a_9,0.3991) & \tilde{\theta}(a_9,0.3991)& \tilde{\alpha}(a_9,0.3991) \\
\tilde{r}(a_{10},0.3991) & \tilde{\theta}(a_{10},0.3991)& \tilde{\alpha}(a_{10},0.3991) \\
\tilde{r}(a_{11},0.3991) & \tilde{\theta}(a_{11},0.3991)& \tilde{\alpha}(a_{11},0.3991) \\
\tilde{r}(a_{12},0.3991) & \tilde{\theta}(a_{12},0.3991)& \tilde{\alpha}(a_{12},0.3991) \\
\tilde{r}(a_{13},0.3991) & \tilde{\theta}(a_{13},0.3991)& \tilde{\alpha}(a_{13},0.3991)  \\
\tilde{r}(a_{14},0.3991) & \tilde{\theta}(a_{14},0.3991)& \tilde{\alpha}(a_{14},0.3991) \\
\tilde{r}(a_{15},0.3991) & \tilde{\theta}(a_{15},0.3991)& \tilde{\alpha}(a_{15},0.3991)  \\
\end{array}
\right)
=\left(
\begin{array}{ccc}
 \frac{13221985897}{10000000000} & \frac{7849465573}{10000000000} & \frac{625614391}{400000000} \\
 \frac{3305341587}{2500000000} & \frac{7851741557}{10000000000} & \frac{1564313037}{1000000000} \\
 \frac{13220747031}{10000000000} & \frac{7854017679}{10000000000} & \frac{7822950733}{5000000000} \\
 \frac{3305031973}{2500000000} & \frac{7856293797}{10000000000} & \frac{15648672607}{10000000000} \\
 \frac{13219508981}{10000000000} & \frac{3929284889}{5000000000} & \frac{15651443943}{10000000000} \\
 \frac{413090323}{312500000} & \frac{1965211471}{2500000000} & \frac{7827107803}{5000000000} \\
 \frac{6609135893}{5000000000} & \frac{3931560893}{5000000000} & \frac{978561719}{625000000} \\
 \frac{13217653441}{10000000000} & \frac{491587361}{625000000} & \frac{7829879799}{5000000000} \\
 \frac{1321703543}{1000000000} & \frac{1966918439}{2500000000} & \frac{15662531883}{10000000000} \\
 \frac{13216417529}{10000000000} & \frac{9837437}{12500000} & \frac{15665304353}{10000000000} \\
 \frac{660789997}{500000000} & \frac{196805639}{250000000} & \frac{15668077229}{10000000000} \\
 \frac{13215182437}{10000000000} & \frac{1574900287}{2000000000} & \frac{3134170029}{2000000000} \\
 \frac{6607282587}{5000000000} & \frac{7876777299}{10000000000} & \frac{15673623461}{10000000000} \\
 \frac{13213948101}{10000000000} & \frac{3939526607}{5000000000} & \frac{7838198359}{5000000000} \\
 \frac{3303332839}{2500000000} & \frac{1576265811}{2000000000} & \frac{7839585261}{5000000000} \\
 \frac{13212714809}{10000000000} & \frac{3941802431}{5000000000} & \frac{15681944331}{10000000000} \\
\end{array}
\right)}
\]

For the estimates above the error provided by the RTM is smaller that

$$\Bigl(\tfrac{M\,h^m}{L\,(m+1)!} + \tfrac{\sqrt n\,R}{L\,h}\Bigr)
    \bigl(e^{Lkh}-1\bigr)<0.0003048$$

Recall that $m=1$, $R=10^{-10}$, $n=3$, $L=K_0=6.8246$ and $M=\sqrt{M_1^2+M_2^2+M_3^2}\approx 16.9424$ and $h$ is either $\frac{0.3991}{25000}$ or $\frac{0.3966}{25000}$. A direct computations shows that all the 16 values $\tilde{\alpha}(a_j,0.3966)$ satisfy the following inequalities

$$\frac{\pi}{2}<\tilde{\alpha}(a_0,0.3966)<\tilde{\alpha}(a_1,0.3966)<\dots<\tilde{\alpha}(a_{15},0.3966)$$

For any $a$ between $a_0=0.5204$ and $a_{15}=0.5244$, there exists a $j$ such that $|a-a_j|\le \frac{0.002}{15}$, therefore, using Theorem \ref{thm:stability}, we obtain that 

$$|\alpha(a,0.3966)-\alpha(a_j,0.3966)|\le \frac{0.002}{15} e^{ 0.3966K_0}<0.001998$$

Since the distance between $\alpha(a_j,0.3966)$ and $\tilde{\alpha}(a_j,0.3966)$ is smaller than $0.0003048$ and the distance between $\tilde{\alpha}(a_j,0.3966)$ and $\frac{\pi}{2}$ is greater than $0.00264$ then we conclude that:

$$ \hbox{for every $a\in [0.5204,0.5244]$, $\alpha(a,0.3966)\ge \frac{\pi}{2}$}$$

A similar argument shows that 

$$ \hbox{for every $a\in [0.5204,0.5244]$, $\alpha(a,0.3991)\le \frac{\pi}{2}$}$$

This time we have that 

$$\tilde{\alpha}(a_0,0.3991)<\tilde{\alpha}(a_1,0.3991)<\dots<\tilde{\alpha}(a_{15},0.3991)<\frac{\pi}{2}\, ,$$

$$|\alpha(a,0.3991)-\alpha(a_j,0.3991)|\le \frac{0.002}{15} e^{ 0.3991K_0}<0.00204\, ,$$

 the distance between $\alpha(a_j,0.3991)$ and $\tilde{\alpha}(a_j,0.3991)$ is smaller than $0.0003048$ and the distance between $\tilde{\alpha}(a_j,0.3991)$ and $\frac{\pi}{2}$ is greater than $0.002601$. Let us move to study the function $\theta$. Let us denote by 
 
 $$\theta_i(0.5204) \quad and \quad\hbox{with $i\, =\, 1,\dots 25000$ }$$
 
 The approximations for $\theta(0.5204,i\frac{0.3991}{25000})$ given by the RTM using $h=\frac{0.3991}{25000}$. Recall that (see table above)
 
 $$\theta_{25000}(0.5204)=\tilde{\theta}(a_0,0.3991)=\frac{7849465573}{10000000000}<\frac{\pi}{4}-0.00045$$
 
 We can check that all the $\theta_i(0.5204)$ are smaller than $\theta_{25000}(0.5204)$ (actually the sequence $\{\theta_i(0.5204)\}_{i=1}^{25000}$ is increasing). 
 
 For any $t\in[0.3966,0.3991]$ there exists and $i$ such that 
 $|t-i\frac{0.3991}{25000}|\le \frac{0.3991}{50000}$. We have that
 
 $$|\theta(0.5204,t)-\theta(0.5204,i\frac{0.3966}{25000})|\le \theta'(0.5204,\bar{t}) \, |t-i\frac{0.3966}{25000}|<1.033*\frac{0.3966}{50000}<0.00001$$
 
 In the inequality above we are using the fact that $|f_3(u)|<1.033$ and that $(0.5244,\bar{t})\in U_2$ because $\bar{t}$ is between $t$ and $ i\frac{0.3966}{25000}$
 
 Using the error formula for the RTM, we get that $|\theta_i(0.5004)-\theta(0.5204,i\frac{0.3966}{25000})|<0.0003048$ and therefore, since the distance from  $\theta_i(0.5204)$ to 
 $\frac{\pi}{4}$ is greater than $0.00045$ we get that

 $$ \hbox{for every $t\in [0.3966,0.3944]$, $\theta(0.5204,t)\le \frac{\pi}{4}$}$$
 
 Let us denote by 
 
 $$\theta_i(0.5244) \quad and \quad\hbox{with $i,1,\dots 25000$ }$$
 
 The approximations for $\theta(0.5244,i\frac{0.3991}{25000})$ given by the RTM using $h=\frac{0.3991}{25000}$. Once again we can check that the sequence $\{\theta_i(0.5244)\}_{i=1}^{25000}$ is increasing. Since
 
 $$24843\frac{0.3991}{25000}<0.3966<24844\frac{0.3991}{25000}\, ,$$
 
 we have that for every  for any $t\in[0.3966,0.3991]$ there exists an integer  $i\in [24843,25000]$ such that 
 $|t-i\frac{0.3991}{25000}|\le \frac{0.3991}{50000}$. Same arguments as before show that 
 
 $$|\theta(0.5244,t)-\theta(0.5244,i\frac{0.3991}{25000})|<0.00001\, .$$ 
 
 Since 
 
 $$\frac{\pi}{4}<\theta_{24843}(0.5244)=\frac{7857740589}{10000000000}<\dots<\theta_{25000}(0.5244)=\frac{3941802431}{5000000000}$$
 
 and the distance between $\theta_{24843}(0.5244)$ and $\frac{\pi}{4}$ is bigger than $0.000375$, we conclude that
 
 $$ \hbox{for every $t\in [0.3966,0.3991]$, $\theta(0.5244,t)\ge \frac{\pi}{4}$}$$

The Poincare-Miranda theorem and the inequalities of the functions$\alpha$ and $\theta$ on the boundary of the rectangle $ [0.5204,0.5244]\times [0.3966,0.3991]$ gives the existence of an 
$(a_*,t_*)$ such that $\alpha(a_*,t_*)=\frac{\pi}{2}$ and $\theta(a_*,t_*)=\frac{\pi}{4}$.

\end{proof}

\end{document}